\documentclass{amsart}

\usepackage{latexsym,amsmath,amstext,mathrsfs,amsthm,amscd,amsopn,verbatim,amsrefs,tikz}
\tikzstyle{n}=[circle, draw, fill, minimum size=6, inner sep=0]
\tikzstyle{ext}=[circle, draw,  minimum size=3, inner sep=0]
\tikzstyle{int}=[circle, draw, fill, minimum size=3, inner sep=0]
\usetikzlibrary{matrix, arrows}

\usepackage{graphicx}
\usepackage[all]{xy}
\usepackage{amscd}
\usepackage{amssymb}
\usepackage{MnSymbol}
\usepackage[english,french]{babel}

\newtheorem{Thm}{Theorem}[section]

\theoremstyle{remark}
\newtheorem{Rem}[Thm]{Remark}

\theoremstyle{definition}

\title[On the compatibility between cup products...]{On the compatibility between cup products, the Alekseev--Torossian connection and the Kashiwara--Vergne conjecture}
\author{C.~A.~Rossi}
\address{MPIM Bonn, Vivatsgasse 7, 53111 Bonn (Germany)}

\begin{document}



\maketitle

\selectlanguage{english}
\begin{abstract}
For a finite-dimensional Lie algebra $\mathfrak g$ over a field $\mathbb K\supset \mathbb C$, we deduce from the compatibility between cup products~\cite[Section 8]{K} and from the main result of~\cite{Sh} an alternative way of re-writing Kontsevich product $\star$ on $\mathrm S(\mathfrak g)$ by means of the Alekseev--Torossian flat connection~\cite{AT}. 
We deduce a similar formula directly from the Kashiwara--Vergne conjecture~\cite{KV}.
\end{abstract}



\selectlanguage{english}

\section{Introduction}\label{s-0}
For a general finite-dimensional Lie algebra $\mathfrak g$ over a field $\mathbb K\supset \mathbb C$, we consider the symmetric algebra $A=\mathrm S(\mathfrak g)$.

Deformation quantization {\em \`a la} Kontsevich~\cite{K} endows $A$ with an associative, non-commutative product $\star$: the universal property of the Universal Enveloping Algebra (shortly, from now on, UEA) $\mathrm U(\mathfrak g)$ and a degree argument imply that there is an isomorphism of associative algebras $\mathcal I$ from $(A,\star)$ to $(\mathrm U(\mathfrak g),\cdot)$.
In fact, the algebra isomorphism $\mathcal I$ has been characterized explicitly in~\cite{AST,CFR} as the composition of the Poincar\'e--Birkhoff--Witt (shortly, from now on, PBW) isomorphism (of vector spaces) with an invertible differential operator with constant coefficients and of infinite order associated to the well-known Duflo element $\sqrt{j(\bullet)}$ in the completed symmetric algebra $\widehat{\mathrm S}(\mathfrak g^*)$.

In this short note, we deduce a way of re-writing the product $\star$ on $A$ in terms of the Lie series $F$, $G$ appearing in the combinatorial Kashiwara--Vergne (shortly, from now on, KV) conjecture~\cite{KV}.
In fact, we prove a similar claim by deducing it from the compatibility between cup products for Kontsevich's formality quasi-isomorphism~\cite[Section 8]{K}.
Both claims are proved in a constructive way (see later on Formul\ae~\eqref{eq-AT-cup},~\eqref{eq-cup-comp} and~\eqref{eq-KV-cup}).

In Section~\ref{s-1}, we quickly review the main notation and conventions.

In Subsection~\ref{ss-2-1}, we recall the main features of Kontsevich's deformation quantization.

In Subsection~\ref{ss-2-2}, we re-prove in a different way the compatibility between cup products in $0$-th degree for the tangent cohomology in the lie algebra case in order to get Formul\ae~\eqref{eq-AT-cup} and~\eqref{eq-cup-comp}.

In Subsection~\ref{ss-2-3}, we re-find the Alekseev--Torossian (shortly, from now on, AT) connection~\cite{T,AT}.
\begin{Rem}\label{r-at}
The results of Subsections~\ref{ss-2-2},\ref{ss-2-3} were already somehow present in the seminal work~\cite{T} of C.~Torossian: we present them here in a different fashion, and the compatibility between cup products in $0$-th cohomology is proved in a different way than in~\cite[Subsection 8.2]{K}. 
\end{Rem}
\begin{Rem}\label{r-m}
The conjecture of Kashiwara--Ra\"is--Vergne (shortly, from now on, KRV) in the framework of point-supported distributions on Lie groups and Lie algebras has been proved in~\cite[section 4]{M} as a consequence of a result extending compatibility between cup products as in~\cite[Subsection 8.2]{K} to a compatibility between certain $A_\infty$-algebra structures on (twisted) poly-vector fields and poly-differential operators over $X=\mathbb K^d$ determined by a choice of a (twisted) poly-vector field of degree $1$ satisfying the Maurer--Cartan equation in the Schouten algebra of poly-vector fields over $X$.

However, the main result~\cite[Identity (55)]{M} does not hold true, because the crucial argument in~\cite[subsection 3.1]{M} is false: in fact,~\cite[Kontsevich's Vanishing Lemma 6.6]{K} does not hold true in the framework of the variations of the compactified configuration spaces considered in~\cite[Subsection 2.3]{M} except for the situation, where one considers only cup products.
This has been proved by direct computations either in~\cite[Example 2]{AT} or in~\cite[Subsection 3.4]{Alm}.
Furthermore, in the latter work, it has been proved explicitly~\cite[Section 4]{Alm} that~\cite[identity (55)]{M}, in the case of a formal Poisson structure $\hbar\pi$ on $X$, holds true only provided one introduces an ``exotic'' $A_\infty$-structure on the Schouten algebra of poly-vector fields over $X$.

Still, the proof of the KRV conjecture in~\cite[Section 4]{M} remains true, as it makes use only of compatibility between cup products, which in turn relies on the special case~\cite[Identity (56)]{M} for cup products of~\cite[Identity (55)]{M}, the latter dealing with the entire $A_\infty$-structure.

We provide here a different proof of the same result: the techniques presented here are more general than the ones adopted in~\cite{M} (although we apply them here only to prove compatibility between cup products) and imply the results of~\cite{Alm} as well, upon the choice of certain homology chains in the relevant compactified configuration spaces {\em \`a la} Kontsevich.
\end{Rem}
Finally, in Subsection~\ref{ss-2-4}, we consider the combinatorial KV conjecture and from it we deduce Formula~\eqref{eq-KV-cup}, which also yields compatibility between cup products in $0$-th cohomology. 

 \subsection*{Acknowledgments} We thank A.~Alekseev and G.~Felder for constructive criticism on the first drafts and for useful discussions, J.~L\"offler for many useful discussions and for the careful reading of a first version of this short note and M.~Vergne for having raised the main question originating this short note and many useful discussions and criticism.

\section{Notation and conventions}\label{s-1}
We consider a field $\mathbb K\supset\mathbb C$.

We denote by $\mathfrak g$ a finite-dimensional Lie algebra over $\mathbb K$ of dimension $d$; by $\{x_i\}$ we denote a $\mathbb K$-basis of $\mathfrak g$.
To $\mathfrak g$ we associate the linear variety $X=\mathfrak g^*$ over $\mathbb K$: the basis $\{x_i\}$ defines a set of global coordinates over $X$, and the Kirillov--Kostant Poisson bivector field $\pi$ on $X$ can be written as $\pi=f_{ij}^kx_k\partial_i\partial_j$, where we have omitted wedge product for the sake of simplicity, and $f_{ij}^k$ denote the structure constants of $\mathfrak g$ w.r.t.\ the basis $\{x_i\}$.

\section{Compatibility between cup products and the AT connection}\label{s-2}
In the present section, we consider a slightly different approach to the compatibility between cup products from~\cite[Subsection 8.2]{K} on the $0$-th cohomology.
We then specialize to the case of the Poisson variety $(X,\pi)$, where $X=\mathfrak g^*$ for $\mathfrak g$ as in Section~\ref{s-1}.

\subsection{Explicit formul\ae\ for Kontsevich's star product}\label{ss-2-1}
Let $X=\mathbb K^d$ and $\{x_i\}$ a system of global coordinates on $X$, for $\mathbb K$ as above.

For a pair $(n,m)$ of non-negative integers, by $\mathcal G_{n,m}$ we denote the set of admissible graphs of type $(n,m)$: an element $\Gamma$ of $\mathcal G_{n,m}$ is a directed graph with $n$, resp.\ $m$, vertices of the first, resp.\ second type, such that $i)$ there is no directed edge departing from any vertex of the second type and $ii)$ $\Gamma$ admits whether multiple edges nor short loops ({\em i.e.} given two distinct vertices $v_i$, $i=1,2$, of $\Gamma$ there is at most one directed edge from $v_1$ to $v_2$ and there is no directed edge, whose endpoint coincides with the initial point). 
By $E(\Gamma)$ we denote the set of edges of $\Gamma$ in $\mathcal G_{n,m}$.

We denote by $C_{n,m}^+$ the configuration space of $n$ points in the complex upper half-plane $\mathbb H^+$ and $m$ ordered points on the real axis $\mathbb R$ {\em modulo} the componentwise action of rescalings and real translations: provided $2n+m-2\geq 0$, $C_{n,m}^+$ is a smooth, oriented manifold of dimension $2n+m-2$.
We denote by $\mathcal C_{n,m}^+$ a suitable compactification {\em \`a la} Fulton--MacPherson introduced in~\cite[Section 5]{K}: $\mathcal C_{n,m}^+$ is a compact, oriented, smooth manifold with corners of dimension $2n+m-2$.
W
We denote by $\omega$ the closed, real-valued $1$-form
\[
\omega(z_1,z_2)=\frac{1}{2\pi}d\mathrm{arg}\!\left(\frac{z_1-z_2}{\overline z_1-z_2}\right),\ (z_1,z_2)\in (\mathbb H^+\sqcup \mathbb R)^2,\ z_1\neq z_2,
\]
where $\mathrm{arg}(\bullet)$ denotes the $[0,2\pi)$-valued argument function on $\mathbb C\smallsetminus\{0\}$ such that $\mathrm{arg}(i)=\pi/2$.
The main feature of $\omega$ is that it extends to a smooth, closed $1$-form on $\mathcal C_{2,0}^+$, such that $i)$ when the two arguments approach to each other in $\mathbb H^+$, $\omega$ equals the normalized volume form $d\varphi$ on $S^1$ and $ii)$ when the first argument approaches $\mathbb R$, $\omega$ vanishes.

We introduce $T_\mathrm{poly}(X)=A[\theta_1,\dots,\theta_d]$, $A=C^\infty(X)$, where $\{\theta_i\}$ denotes a set of graded variables of degree $1$, which commute with $A$ and anticommute with each other (one may think of $\theta_i$ as $\partial_i$ with a shifted degree).
We further consider the well-defined linear endomorphism $\tau$ of $T_\mathrm{poly}(X)^{\otimes 2}$ of degree $-1$ defined {\em via}
\[
\tau=\partial_{\theta_i}\otimes\partial_{x_i},
\]
where of course summation over repeated indices is understood.
We set $\omega_\tau=\omega\otimes \tau$.

To $\Gamma$ in $\mathcal G_{n,m}$ such that $|E(\Gamma)|=2n+m-2$, $\gamma_i$, $i=1,\dots,n$, elements of $T_\mathrm{poly}(X)$ and $a_j$, $j=1,\dots,m$, elements of $A$, we associate a map {\em via}
\begin{equation}\label{eq-form-coeff}
\left(\mathcal U_\Gamma(\gamma_1,\dots,\gamma_n)\right)(a_1\otimes\cdots\otimes a_m)=\mu_{m+n}\left(\int_{\mathcal C_{n,m}^+}\omega_{\tau,\Gamma}\left(\gamma_1\otimes\cdots\otimes\gamma_n\otimes a_1\otimes\cdots\otimes a_m\right)\right),\ \omega_{\tau,\Gamma}=\prod_{e\in E(\Gamma)}\omega_{\tau,e},\ \omega_{\tau,e}=\pi_e^*(\omega)\otimes\tau_e,
\end{equation}
$\tau_e$ being the graded endomorphism of $T_\mathrm{poly}(X)^{\otimes(m+n)}$ which acts as $\tau$ on the two factors of $T_\mathrm{poly}(X)$ corresponding to the initial and final point of the edge $e$, and $\mu_{m+n}$ denotes the multiplication map from $T_\mathrm{poly}(X)^{m+n}$ to $T_\mathrm{poly}(X)$, followed by the natural projection from $T_\mathrm{poly}(X)$ onto $A$ by setting $\theta_i=0$, $i=1,\dots,d$.
We may re-write~\eqref{eq-form-coeff} by splitting the form-part and the polydifferential operator part as
\[
\left(\mathcal U_\Gamma(\gamma_1,\dots,\gamma_n)\right)(a_1\otimes\cdots\otimes a_m)=\varpi_\Gamma(\mathcal B_\Gamma(\gamma_1,\dots,\gamma_n))(a_1,\dots,a_m),\ \varpi_\Gamma=\int_{\mathcal C_{n,m}^+}\omega_\Gamma.
\]

In~\cite[Theorem 6.4]{K}, the following theorem has been proved.
\begin{Thm}\label{t-star}
For a Poisson bivector field $\pi$ on $X$, and a formal parameter $\hbar$, the formula
\begin{equation}\label{eq-star}
f_1\star_{\hbar} f_2=\sum_{n\geq 0}\frac{\hbar^n}{n!}\sum_{\Gamma\in\mathcal G_{n,2}}(\mathcal U_\Gamma(\underset{n}{\underbrace{\pi,\dots,\pi}}))(f_1,f_2),\ f_i\in A,\ i=1,2,
\end{equation}
defines a $\mathbb K_{\hbar}=\mathbb K[\!\![\hbar]\!\!]$-linear, associative product on $A_{\hbar}=A[\!\![\hbar]\!\!]$. 
\end{Thm}

\subsection{The $1$-form governing the compatibility between cup products}\label{ss-2-2}
We now consider $\mathfrak g$ as in Section~\ref{s-1}, to which we associate the Poisson variety $(X=\mathfrak g^*,\pi)$.
Observe that the commutative algebra $\mathbb K[X]$ of regular functions on $X$ identifies with $A=\mathrm S(\mathfrak g)$.

Since $\pi$ is linear, Formula~\eqref{eq-star} restricts to $A_{\hbar}$ and moreover the $\hbar$-dependence is polynomial: we may thus safely set $\hbar=1$ and consider the associative algebra $(A,\star)$.

For a non-negative integer $n$, let us consider the projection $\pi_{n,2}$ from $C_{n+2,0}^+$ onto $C_{2,0}^+$ which forgets all points in $\mathbb H^+$ except the last two: it extends smoothly to a projection from $\mathcal C_{n+2,0}^+$ onto $\mathcal C_{2,0}^+$, which we denote by the same symbol.
It is clear that $\pi_{n,2}$ defines a fibration onto $\mathcal C_{2,0}^+$, whose typical fiber is a smooth, oriented manifold with corners of dimension $2n$.

To $\Gamma$ in $\mathcal G_{n+2,0}$ such that $|E(\Gamma)|=2n$, we associate a smooth $0$-form on $C_{2,0}^+$ with values in the bidifferential operators on $A$ defined as 
\begin{equation}\label{eq-tang-coeff}
\mathcal T_\Gamma^\pi(f_1,f_2)=\mu_{n+2}(\pi_{n,2,*}(\omega_{\tau,\Gamma}(\underset{n}{\underbrace{\pi\otimes\cdots\otimes\pi}}\otimes f_1\otimes f_2)))=\widehat\varpi_\Gamma(\mathcal B_\Gamma(\underset{n}{\underbrace{\pi,\dots,\pi}}))(f_1,f_2),\ \widehat\varpi_\Gamma=\pi_{n,2,*}(\omega_\Gamma),
\end{equation}
where $\pi_{n,2,*}$ denotes the integration along the fiber of the operator-valued form $\omega_{\tau,\Gamma}$ w.r.t.\ the projection $\pi_{n,2}$.
We finally set
\begin{equation}\label{eq-tang-func}
\mathcal T^\pi(f_1,f_2)=\sum_{n\geq 0}\frac{1}{n!}\sum_{\Gamma\in \mathcal G_{n+2,0}\atop |E(\Gamma)|=2n}\mathcal T_\Gamma^\pi(f_1,f_2),\ f_i\in A,\ i=1,2.
\end{equation}
Formula~\eqref{eq-tang-func} yields a well-defined smooth function on $C_{2,0}^+$ with values in the bidifferential operators on $A$.
\begin{Thm}\label{t-AT-cup}
There exist smooth $1$-forms $\Omega_i^\pi$ on $C_{2,0}^+$, $i=1,2$, with values in $\mathfrak g\otimes \widehat{\mathrm S}(\mathfrak g^*)^{\otimes 2}$, such that the following identity holds true:
\begin{equation}\label{eq-AT-cup}
d(\mathcal T^\pi(f_1,f_2))=\mathcal T^\pi(\Omega_1^\pi([\pi,f_1],f_2))+\mathcal T^\pi(\Omega_2^\pi(f_1,[\pi,f_2])),\ f_i\in A,\ i=1,2.
\end{equation}   
\end{Thm}
\begin{proof}
First of all, for $\Gamma$ in $\mathcal G_{n+2,0}$ such that $|E(\Gamma)|=2n$, $n\geq 1$, let us compute
\[
d(\mathcal T_\Gamma^\pi(f_1,f_2))=d\widehat\varpi_\Gamma (\mathcal B_\Gamma(\underset{n}{\underbrace{\pi,\dots,\pi}}))(f_1,f_2)=\pi_{n,2,*}^\partial(\omega_\Gamma)(\mathcal B_\Gamma(\underset{n}{\underbrace{\pi,\dots,\pi}}))(f_1,f_2),
\]
where the second equality follows by means of the generalized Stokes Theorem for integration along the fiber, and $\pi_{n,2,*}^\partial$ denotes integration along the boundary of the compactification of the typical fiber of the projection $\pi_{n,2}$.

The boundary strata of codimension $1$ of the compactification of the typical fiber of $\pi_{n,2}$ can be deduced from the boundary strata of codimension $1$ of $\mathcal C_{n+2,0}^+$: 
\begin{itemize}
\item[$i)$] there is a subset $A$ of $[n+2]=\{1,\dots,n+2\}$, $1\leq |A|\leq n$ which contains either $n+1$ or $n+2$, such that points in $\mathbb H^+$ labeled by $A$ collapse either to the $n+1$-st or $n+2$-nd point in $\mathbb H^+$;
\item[$ii)$] there is a subset $A$ of $[n+2]$, $2\leq |A|\leq n$, $n+1,n+2\notin A$, such that points in $\mathbb H^+$ labeled by $A$ collapse to a single point in $\mathbb H^+$, distinct from the last two points;
\item[$iii)$] there is a subset $A$ of $[n+2]$, which either contains both $n+1$, $n+2$ or contains neither of them, such that the points in $\mathbb H^+$ labeled by $A$ approach $\mathbb R$.
\end{itemize}

For $\Gamma$ as above, we denote by $\Gamma_A$ the subgraph of $\Gamma$, whose edges have both endpoints labeled by $A$, and by $\Gamma/\Gamma_A$ the corresponding quotient graph obtained by shrinking $\Gamma_A$ in $\Gamma$ to a single vertex.

The boundary strata of type $iii)$ yield trivial contributions.
Namely, let us consider first a subset $A$ of $[n+2]$, such that $n+1, n+2\notin A$: Fubini's Theorem implies that
\[
\pi_{n,2,*}^{\partial,A}(\omega_\Gamma)\varpropto\int_{\mathcal C_{A,0}^+}\omega_{\Gamma_A},
\]
and the aforementioned properties of $\omega$ imply that the form degree of $\omega_{\Gamma_A}$ equals $2|A|$, while the dimension of $C_{A,0}^+$ equals $2|A|-2$.
If $A$ contains both $n+1$, $n+2$, we may repeat the previous arguments {\em verbatim} by replacing $A$ by $A^c$.

Let us consider a general boundary stratum of type $ii)$: Fubini's Theorem and the properties of $\omega$ imply
\[
\pi_{n,2,*}^{\partial,A}(\omega_{\Gamma})\varpropto \int_{\mathcal C_A}\omega_{\Gamma_A},
\]
where $\mathcal C_A$ is the compactified configuration space of $|A|$ points in $\mathbb C$ {\em modulo} rescalings and complex translations; by abuse of notations, we have denoted by $\omega_{\Gamma_A}$ a product of $1$-forms $d\arg(z_i-z_j)$, $i$, $j$ in $A$, on $\mathcal C_A$.
If $|A|\geq 3$, the above integral on the right-hand side vanishes by~\cite[Lemma 6.6]{K}.
Thus, it remains to consider the case $|A|=2$.
If no edge connects the two vertices labeled by $A$, there is nothing to integrate over $\mathcal C_2=S^1$, while, if there is a cycle between the two vertices, $\omega_{\Gamma_A}$ is the square of a $1$-form, hence both contributions vanish.
We thus assume that there is a single edge connecting the two vertices labeled by $A$, in which case Fubini's Theorem together with the properties of $\omega$ when its arguments collapse in $\mathbb H^+$ yields 
\[
\pi_{n,2,*}^{\partial,A}(\omega_{\Gamma})=\pi_{n-1,2,*}(\omega_{\Gamma/\Gamma_A}).
\]
Observe that $\Gamma/\Gamma_A$ belongs to $\mathcal G_{n+1,0}$, no edge departs from $n+1$, $n+2$ and all other vertices are bivalent except one, which is trivalent (here, the valence of a vertex is the number of outgoing edges from the said vertex).

Finally, let us consider a boundary stratum of type $i)$, labeled by a subset $A$ of $[n+2]$, such that $n+1\in A$, $n+2\notin A$.
Assume first $|A|\geq 2$: then, in a way similar to the analysis of a boundary stratum of type $ii)$, we find
\[
\pi_{n,2,*}^{\partial,A}(\omega_{\Gamma})\varpropto \int_{\mathcal C_{A\sqcup\{n+1\}}}\omega_{\Gamma_A}=0
\]
by~\cite[Lemma 6.6]{K}, as $|A|\geq 2$.
It remains to consider the case $|A|=1$.
As before, we may safely assume that $\Gamma_A$ consists of a single edge with endpoint $n+1$ and initial point different from $n+2$, whence
\[
\pi_{n,2,*}^{\partial,A}(\omega_{\Gamma})=\pi_{n-1,2,*}(\omega_{\Gamma/\Gamma_A}).
\]
Due modifications of the previous arguments yield a similar formula in the situation $n+1\notin A$, $n+2\in A$. 
Here, $\Gamma/\Gamma_A$, if $n+1$ is in $A$, belongs to $\mathcal G_{n+1,0}$, exactly one edge departs from $n+1$, no edge departs from $n+2$, and all other vertices are bivalent; when $n+2$ belongs to $A$, $\Gamma/\Gamma_A$ is described in a similar way by switching $n+1$ and $n+2$.

The previous computations yield
\begin{equation}\label{eq-cup-form}
\begin{aligned}
d(\mathcal T^\pi(f_1,f_2))&=\sum_{n\geq 1}\frac{1}{n!}\sum_{\Gamma\in\mathcal G_{n+2,0}\atop |E(\Gamma)|=2n}\sum_{A\subseteq [n+2],\ |A|=2\atop n+1\in A,\ n+2\notin A} \widehat\varpi_{\Gamma/\Gamma_A}(\mathcal B_\Gamma(\underset{n}{\underbrace{\pi,\dots,\pi}}))(f_1,f_2)+\sum_{n\geq 1}\frac{1}{n!}\sum_{\Gamma\in\mathcal G_{n+2,0}\atop |E(\Gamma)|=2n}\sum_{A\subseteq [n+2],\ |A|=2\atop n+1\notin A,\ n+2\in A} \widehat\varpi_{\Gamma/\Gamma_A}(\mathcal B_\Gamma(\underset{n}{\underbrace{\pi,\dots,\pi}}))(f_1,f_2)+\\
&\phantom{=}+\sum_{n\geq 1}\frac{1}{n!}\sum_{\Gamma\in\mathcal G_{n+2,0}\atop |E(\Gamma)|=2n}\sum_{A\subseteq [n+2],\ |A|=2\atop n+1,n+2\notin A} \widehat\varpi_{\Gamma/\Gamma_A}(\mathcal B_\Gamma(\underset{n}{\underbrace{\pi,\dots,\pi}}))(f_1,f_2)=\\
&=\sum_{n\geq 0}\frac{1}{n!}\sum_{\Gamma\in\mathcal G_{n+2,0}\atop |E(\Gamma)|=2n+1}\widehat{\varpi}_\Gamma(\mathcal B_\Gamma(\underset{n}{\underbrace{\pi,\dots,\pi}}))([\pi,f_1],f_2)+\sum_{n\geq 0}\frac{1}{n!}\sum_{\Gamma\in\mathcal G_{n+2,0}\atop |E(\Gamma)|=2n+1}\widehat\varpi_\Gamma(\mathcal B_\Gamma(\underset{n}{\underbrace{\pi,\dots,\pi}}))(f_1,[\pi,f_2])+\\
&\phantom{=}+\sum_{n\geq 0}\frac{1}{n!}\sum_{\Gamma\in\mathcal G_{n+2,0}\atop |E(\Gamma)|=2n+1}\widehat\varpi_\Gamma(\mathcal B_\Gamma([\pi,\pi],\underset{n-1}{\underbrace{\pi,\dots,\pi}}))(f_1,f_2),
\end{aligned}
\end{equation}
recalling the explicit shape of the quotient subgraph $\Gamma/\Gamma_A$ in the three previous cases and using Leibniz' rule to re-write the sums over $A$ in the bidifferential operators; $[\pi,\pi]$ denotes the trivector field on $x$, whose components are given by the sum over the cyclic permutations of $\{j,k,l\}$ in $\pi_{ij}\partial_i\pi_{kl}$.

Observe that the third term in the final expression of~\eqref{eq-cup-form} vanishes because of the Jacobi identity, hence only the first and second term matter in our discussion.
The linearity of $\pi$ on $X=\mathfrak g^*$ permits to re-write both $1$-forms in a more elegant way.

If we consider a general graph $\Gamma$ in $\mathcal G_{n+2,0}$ as in the first term on the right-hand side of~\eqref{eq-cup-form}, any bivalent vertex different from $n+1$, $n+2$ may be the endpoint of at most one arrow.
Thus, by slightly adapting the arguments of~\cite[Subsections 3.1.2-3.1.4]{AST}, $\Gamma$ factorizes uniquely into the union of its simple components~\footnote{An element $\Gamma$ of $\mathcal G_{n,2}$ is simple, if the graph obtained from $\Gamma$ by removing all arrows connecting to the vertices of the second type is connected. In the present situation, we may regard $\Gamma$ in $\mathcal G_{n+2,0}$ as an element of $\mathcal G_{n,2}$ by interpreting the last two vertices of the first type as vertices of the second type.}: the main novelty is that in the present situation, there are three types of simple components, namely
\begin{itemize}
\item[$i)$] rooted, bivalent trees with $2$ leaves,
\item[$ii)$] wheel-like graphs with $2$ leaves, whose legs may be attached to rooted, bivalent trees,
\item[$iii)$] rooted, bivalent trees with $2$ leaves and an edge connecting either one of the two leaves to the root. 
\end{itemize} 
In Figure~\ref{fig-1} are depicted three types of simple graphs as in $i)$-$iii)$: the two gray-shaded vertices of the first type are called external, while the remaining vertices of the first type are called internal.
Observe that the external vertices are only endpoints of edges, while the internal vertices have exactly two outgoing edges and one ingoing edge, except the root in $i)$.  
\begin{figure}
\centering
\begin{tikzpicture}[>=latex]
\tikzstyle{k-int}=[draw,fill=gray!40,circle,inner sep=0pt,minimum size=2.5mm]
\tikzstyle{n-int}=[draw,fill=black,circle,inner sep=0pt,minimum size=2.5mm]
\tikzstyle{ext}=[draw,fill=white,circle,inner sep=0pt,minimum size=2.5mm]

\begin{scope}[scale=0.75]
\node[k-int] (1) at (0,0) {};
\node[k-int] (2) at (2,0) {};
\node[n-int] (v1) at (1,1) {};
\node[n-int] (v2) at (0,2) {};
\node[n-int] (v3) at (2,2) {};
\node[n-int] (r) at (1,3) {};
\draw[thick,->] (v1) to (1);
\draw[thick,->] (v1) to (2);
\draw[thick,->] (v2) to (1);
\draw[thick,->] (v2) to (v1);
\draw[thick,->] (v3) to (2);
\draw[thick,->] (v3) to (v2);
\draw[thick,->] (r) to (v2);
\draw[thick,->] (r) to (v3);
\end{scope}

\begin{scope}[scale=0.75,shift={(5,0)}]
\node[k-int] (1) at (0,0) {};
\node[k-int] (2) at (2,0) {};
\begin{scope}[shift={(1,2)}]
\node[n-int] (s1) at (0:1) {};
\node[n-int] (s2) at (90:1) {};
\node[n-int] (s3) at (180:1) {};
\node[n-int] (s4) at (270:1) {};
\end{scope} 
\node[n-int] (v) at (0,1) {};
\draw[thick,->] (v) to (1);
\draw[thick,->] (v) to (2);
\draw[thick,->] (s1) to (s4);
\draw[thick,->] (s4) to (s3);
\draw[thick,->] (s3) to (s2);
\draw[thick,->] (s2) to (s1);
\draw[thick,->] (s1) to (2);
\draw[thick,->] (s2) to (2);
\draw[thick,->] (s3) to (v);
\draw[thick,->] (s4) to (1);
\end{scope}

\begin{scope}[scale=0.75,shift={(10,0)}]
\node[k-int] (1) at (0,0) {};
\node[k-int] (2) at (2,0) {};
\node[n-int] (v1) at (1,1) {};
\node[n-int] (v2) at (0,2) {};
\node[n-int] (v3) at (2,2) {};
\node[n-int] (r) at (1,3) {};
\draw[thick,->] (v1) to (1);
\draw[thick,->] (v1) to (2);
\draw[thick,->] (v2) to (1);
\draw[thick,->] (v2) to (v1);
\draw[thick,->] (v3) to (2);
\draw[thick,->] (v3) to (v2);
\draw[thick,->] (r) to (v2);
\draw[thick,->] (r) to (v3);
\draw[thick,->] (1) [out=145, in=180] to (r);
\end{scope}
\node at (0.75,-0.5) {$i)$};
\node at (4.5,-0.5) {$ii)$};
\node at (8.25,-0.5) {$iii)$};

\end{tikzpicture}
\caption{\label{fig-1} \selectlanguage{english} $i)$ A rooted, bivalent tree in $\mathcal G_{6,0}$, $ii)$ a wheel-like graph with a bivalent root tree in $\mathcal G_{7,0}$, $iii)$ a rooted, bivalent tree in $\mathcal G_{6,0}$ with an edge connecting the first external vertex to the root.
}
\end{figure}
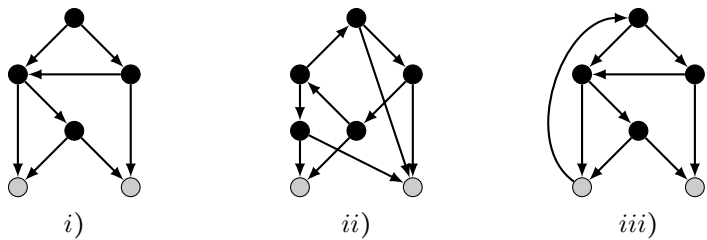
By definition, $\Gamma$ has exactly one simple component of type $iii)$.

Let us consider a simple graph ${}_{\Rsh}\Gamma$, resp.\ $\Gamma_{\Lsh}$, of type $iii)$ with exactly one edge connecting $n+1$, resp.\ $n+2$, to the root: then, borrowing previous notation, we may define
\begin{align}
\label{eq-AT-1}\Omega_{1,{}_{\Rsh}\Gamma}^\pi(f_1\otimes \xi,f_2)=\widehat\varpi_{{}_{\Rsh}\Gamma}\left((\xi\otimes 1\otimes 1)\circ(\mu_n\otimes 1\otimes 1)\circ \tau_\Gamma\right)(\underset{n}{\underbrace{\pi\otimes\cdots\otimes\pi}}\otimes f_1\otimes f_2),\\
\label{eq-AT-2}\Omega_{2,\Gamma_\Lsh}^\pi(f_1,f_2\otimes\xi)=\widehat\varpi_{\Gamma_\Lsh}\left((\xi\otimes 1\otimes 1)\circ(\mu_n\otimes 1\otimes 1)\circ \tau_\Gamma\right)(\underset{n}{\underbrace{\pi\otimes\cdots\otimes\pi}}\otimes f_1\otimes f_2),
\end{align}
where $f_i$ in $A$, $i=1,2$, $\xi$ in $\mathfrak g^*$, and $\Gamma$ is the rooted, bivalent tree obtained from ${}_\Rsh\Gamma$ or $\Gamma_\Lsh$ by removing the edge from $n+1$ or $n+2$ to the root.

Observe that $\widehat\varpi_{{}_\Rsh\Gamma}$ and $\widehat\varpi_{\Gamma_\Lsh}$ are well-defined, smooth $1$-forms on $C_{2,0}^+$.
Further, since $\Gamma$ is a rooted, bivalent tree, $(\mu_n\otimes 1\otimes 1)\circ\tau_\Gamma$ is a linear map from $A^{\otimes 2}$ to $\mathfrak g\otimes A^{\otimes 2}$: hence, contraction of $\mathfrak g$ with $\mathfrak g^*$ yields an endomorphism of $A^{\otimes 2}$ consisting of differential operators with constant coefficients (and possibly infinite order).
Summing up over all simple graphs of type $iii)$~\eqref{eq-AT-1} and~\eqref{eq-AT-2} we obtain well-defined, smooth $1$-forms $\Omega_i^\pi$, $i=1,2$, on $C_{2,0}^+$ with values in $\mathfrak g\otimes\widehat{\mathrm S}(\mathfrak g^*)^{\otimes 2}$, where we identify $\widehat{\mathrm S}(\mathfrak g^*)$ with the algebra of differential operators on $A$ with constant coefficients.

On the other hand, the sum over all simple graphs of type $i)$ and $ii)$ yield the bidifferential operator $\mathcal T_\pi(\bullet,\bullet)$ by the arguments of~\cite[Subsubsections 3.1.2-3.1.4]{AST}.
(We will come back to the simple graphs of type $i)$ and $ii)$ in Subsection~\ref{ss-2-4} about the KV conjecture, where their relevance will be clearer.)

Therefore, Fubini's Theorem and the decomposition of admissible graphs into simple components of type $i)$, $ii)$ and $iii)$ yield~\eqref{eq-AT-cup}. 
\end{proof}
The $0$-form $\mathcal T^\pi$ and the $1$-forms $\Omega_i^\pi$, $i=1,2$, are smooth on $C_{2,0}^+$ and extend to the class $L^1$ when restricted on piecewise differentiable curves on $\mathcal C_{2,0}^+$.

Now, let us evaluate $\mathcal T^\pi(f_1,f_2)$ at a point in the boundary stratum $\mathcal C_2=S^1$ of $\mathcal C_{2,0}^+$, corresponding to the situation, where the two distinct points in $\mathbb H^+$ collapse together along a prescribed direction: the skew-symmetry of $\pi$ eliminates all contributions coming from simple graphs of type $i)$ and of type $ii)$, where at least one rooted, bivalent tree is attached to a wheel-like graph. 
The only possibly non-trivial contributions come from wheel-like graphs with the spokes pointing inwards (the two leaves have collapsed to a single point in $\mathbb H^+$, which we may fix to $i$): the corresponding integral weights vanish by the famous result of~\cite{Sh}.
The only non-trivial contribution comes from the unique graph in $\mathcal G_{2,0}$ with no edges.

Let us evaluate $\mathcal T^\pi(f_1,f_2)$ at the boundary stratum $\mathcal C_{0,2}^+=\{0,1\}$ of codimension $2$ of $\mathcal C_{2,0}^+$, which corresponds to the approach of the two distinct points in $\mathbb H^+$ to $0$ and $1$ on $\mathbb R$: resorting to local coordinates on $\mathcal C_{2,0}^+$ near the said boundary stratum and recalling the projection $\pi_{n,2}$, the corresponding integral weights factorize as $\widehat{\varpi}_\Gamma=\varpi_{\Gamma_1}\varpi_{\Gamma_2}\varpi_{\Gamma_3}$, where $\Gamma_1$ is in $\mathcal G_{n_1,2}$, $\Gamma_2$, $\Gamma_3$ are in $\mathcal G_{n_2,0}$ and $\mathcal G_{n_3,0}$.
Dimensional reasons and the linearity of $\pi$ force $\Gamma_2$ and $\Gamma_3$ to be wheel-like graphs with spokes pointing inwards, thus again in virtue of~\cite{Sh}, the corresponding weights are non-trivial only if $n_2=n_3=0$.

If we consider a piecewise differentiable curve $\gamma$ on $\mathcal C_{2,0}^+$ connecting the said point in $\mathcal C_2=S^1$ with $\mathcal C_{0,2}^+=\{0,1\}$ and whose interior is in $C_{2,0}^+$, we may integrate~\eqref{eq-AT-cup} along $\gamma$: the previous arguments yield
\begin{equation}\label{eq-cup-comp}
f_1\star f_2-f_1f_2=\int_\gamma \left(\mathcal T^\pi(\Omega_1^\pi([\pi,f_1],f_2))+\mathcal T^\pi(\Omega_2^\pi(f_1,[\pi,f_2]))\right),
\end{equation}
which is precisely a special case of the famous compatibility between cup products~\cite[Theorem 8.2]{K}.

\subsection{Relationship with the AT connection}\label{ss-2-3}
By their very construction, $\mathcal T^\pi$ and $\Omega_i^\pi$, $i=1,2$, extend to the completed symmetric algebra $\widehat A=\widehat{\mathrm S}(\mathfrak g)=\mathbb K[\!\![x_1,\dots,x_d]\!\!]$.
For $y_i$, $i=1,2$, in $\mathfrak g$, we consider $e^{y_i}$ in $\widehat A$: $e^{y_i}$ may be also regarded as a smooth function on $X$ {\em via} $e^{y_i}(\xi)=e^{\langle\xi,y_i\rangle}$, $\xi$ in $X$, and $\langle\bullet,\bullet\rangle$ denotes the canonical duality pairing between $\mathfrak g^*$ and $\mathfrak g$.

First of all, borrowing previous notation, let us compute the symbol of $\Omega_i^\pi$, $i=1,2$, {\em i.e.}  
\[
\Omega_1^\pi(e^{y_1}\otimes \xi,e^{y_2}),\ \Omega_1^\pi(e^{y_1},e^{y_2}\otimes\xi),\ \xi\in\mathfrak g^*.
\]
Recalling Formul\ae~\eqref{eq-AT-1},~\eqref{eq-AT-2}, a direct computation yields
\[
\Omega_1^\pi(e^{y_1}\otimes \xi,e^{y_2})=\langle \xi,\omega_1(y_1,y_2)\rangle e^{y_1}\otimes e^{y_2},\ \Omega_2^\pi(e^{y_1}\otimes \xi,e^{y_2})=\langle \xi,\omega_2(y_1,y_2)\rangle e^{y_1}\otimes e^{y_2},
\]
where $\omega_i$ denotes here the AT connection~\cite{T,AT}.
In fact, $\omega_i(y_1,y_2)$, $i=1,2$, denotes a $1$-form on $C_{2,0}^+$ with values in the formal Lie series w.r.t.\ $y_i$ in $\mathfrak g$.
In a more precise way, the AT connection $\omega_i$, $i=1,2$, is a connection $1$-form on $C_{2,0}^+$ with values in the Lie algebra $\mathfrak{tder}_2$ of tangential derivations of the degree completion of the free Lie algebra $\mathfrak{lie}_2$ with two generators~\footnote{A derivation of $\mathfrak{lie}_2$ is uniquely defined on the generators $y_1$, $y_2$: thus, a derivation $u$ of $\mathfrak{lie}_2$ is called tangential, if it obeys $u(y_i)=[y_i,u_i]$, for $u_i$ in $\mathfrak{lie}_2$, $i=1,2$.}. 

Following the same patterns, it is not difficult to prove by direct computations the following identities:
\[
\begin{aligned}
\Omega_1^\pi([\pi,e^{y_1}],e^{y_2})&=\langle[y_1,\omega_1(y_1,y_2)],\partial_{y_1}\rangle(e^{y_1})\otimes e^{y_2}+\mathrm{tr}_\mathfrak g\!\left(\mathrm{ad}(y_1)\partial_{y_1}\omega_1(y_1,y_2)\right)e^{y_1}\otimes e^{y_2},\\
\Omega_2^\pi([\pi,e^{y_1}],e^{y_2})&=e^{y_1}\otimes\langle[y_1,\omega_2(y_1,y_2)],\partial_{y_2}\rangle(e^{y_2})+\mathrm{tr}_\mathfrak g\!\left(\mathrm{ad}(y_2)\partial_{y_2}\omega_2(y_1,y_2)\right)e^{y_1}\otimes e^{y_2},
\end{aligned}
\]
where $\mathrm{tr}_\mathfrak g(\bullet)$ denotes the trace of endomorphisms of $\mathfrak g$, $\mathrm{ad}(\bullet)$ the adjoint representation of $\mathfrak g$ and $\partial_{y_1}\omega_1(y_1,y_2)$ the endomorphism of $\mathfrak g$ defined {\em via}
\[
\left(\partial_{y_1}\omega_1(y_1,y_2)\right)(x)=\frac{d}{dt}\omega_1(y_1+tx,y_2)\Big\vert_{t=0},\ x\in \mathfrak g.
\]

It is possible to re-write~\eqref{eq-cup-comp} as
\[
\begin{aligned}
e^{y_1}\star e^{y_2}-e^{y_1}e^{y_2}&=\int_\gamma\left(\mathcal T^\pi(\langle[y_1,\omega_1(y_1,y_2)],\partial_{y_1}\rangle(e^{y_1}),e^{y_2})+\mathcal T^\pi(e^{y_1},\langle[y_1,\omega_2(y_1,y_2)],\partial_{y_2}\rangle(e^{y_2}))\right)+\\
&\phantom{=}+\left(\mathrm{tr}_\mathfrak g\!\left(\mathrm{ad}(y_1)\partial_{y_1}\omega_1(y_1,y_2)\right)+\mathrm{tr}_\mathfrak g\!\left(\mathrm{ad}(y_2)\partial_{y_2}\omega_2(y_1,y_2)\right)\right)\int_\gamma\mathcal T^\pi(e^{y_1},e^{y_2})=\\
&=\int_\gamma\left(\langle[y_1,\omega_1(y_1,y_2)],\partial_{y_1}\rangle+\langle[y_2,\omega_2(y_1,y_2)],\partial_{y_2}\rangle+\mathrm{div}(\omega(y_1,y_2))\right)D_\mathrm T(y_1,y_2)e^{Z_T(y_1,y_2)},
\end{aligned}
\]
where $\langle[y_1,\omega_1(y_1,y_2)],\partial_{y_1}\rangle(e^{y_1})$ denotes the tangent vector field $[y_1,\omega_1(y_1,y_2)]$ of the adjoint type acting on $e^{y_1}$, and similarly for $\langle[y_2,\omega_2(y_1,y_2)],\partial_{y_2}\rangle(e{y_2})$, and, following notation from~\cite{AT}, 
\[
\mathrm{div}(\omega(y_1,y_2))=\mathrm{tr}_\mathfrak g\!\left(\mathrm{ad}(y_1)\partial_{y_1}\omega_1(y_1,y_2)\right)+\mathrm{tr}_\mathfrak g\!\left(\mathrm{ad}(y_2)\partial_{y_2}\omega_1(y_1,y_2)\right).
\]
Finally, by $D_\mathrm T(\bullet,\bullet)$ and $Z_\mathrm T(\bullet,\bullet)$ we denote the functions over $C_{2,0}^+$, providing deformations of the Duflo density function $D(\bullet,\bullet)$ and the Baker--Campbell--Hausdorff (shortly, BCH) formula $Z(\bullet,\bullet)$ respectively, introduced in~\cite{T}.
We will discuss the Duflo density function in the next Subsection, as well as its relationship with the product $\star$.

\subsection{Relationship with the KV conjecture}\label{ss-2-4}
The AT connection had been introduced in~\cite{T} in an attempt to solve the combinatorial KV conjecture~\cite{KV}.

Given $\mathfrak g$ as in Section~\ref{s-1}, the KV conjecture states the existence of two Lie series $F$, $G$, which are convergent in a neighborhood $U$ of $(0,0)$ in $\mathfrak g\times\mathfrak g$, which satisfy the two identities
\begin{align}
\label{eq-KV-1}y_1+y_2-\log(e^{y_2} e^{y_1})&=\left(1-e^{-\mathrm{ad}(y_1)}\right)F(y_1,y_2)+\left(e^{\mathrm{ad}(y_2)}-1\right)G(y_1,y_2),\\
\label{eq-KV-2}\mathrm{tr}_\mathfrak g(\mathrm{ad}(y_1)\partial_{y_1} F(y_1,y_2))+\mathrm{tr}_\mathfrak g(\mathrm{ad}(y_2)\partial_{y_2} G(y_1,y_2))&=\frac{1}2\mathrm{tr}_\mathfrak g\!\left(\frac{\mathrm{ad}(y_1)}{e^{\mathrm{ad}(y_1)}-1}+\frac{\mathrm{ad}(y_2)}{e^{\mathrm{ad}(y_2)}-1}-\frac{\mathrm{ad}(Z(y_1,y_2))}{e^{\mathrm{ad}(Z(y_1,y_2))}-1}-1\right),
\end{align}
for $(y_1,y_2)$ in $U$, such that the BCH Lie series $Z(y_1,y_2)=\log(e^{y_1}e^{y_2})$ converges. 

We recall from~\cite{CFR} the relationship among the product $\star$ and the product in $\mathrm U(\mathfrak g)$,
\begin{equation}\label{eq-star-UEA}
\mathcal I(f_1\star f_2)=\mathcal I(f_1)\cdot \mathcal I(f_2),\ f_i\in A,\ i=1,2,
\end{equation}
where $\mathcal I$ is the isomorphism (of vector spaces) from $A$ to $\mathrm U(\mathfrak g)$ given by post-composing the PBW isomorphism from $A$ to $\mathrm U(\mathfrak g)$ with the automorphism of $A$ associated to the Duflo function in $\widehat{\mathrm S}(\mathfrak g^*)$,
\[
\sqrt{j(x)}=\sqrt{\mathrm{det}_\mathfrak g\!\left(\frac{1-e^{-\mathrm{ad}(x)}}{\mathrm{ad}(x)}\right)},\ x\in\mathfrak g.
\]

As a corollary of~\eqref{eq-star-UEA}, we have the identity
\begin{equation}\label{eq-star-UEA-1}
e^{y_1}\star e^{y_2}=D(y_1,y_2)e^{Z(y_1,y_2)},\ D(y_1,y_2)=\frac{\sqrt{j(y_1)}\sqrt{j(y_2)}}{\sqrt{j(Z(y_1,y_2))}},\ y_i\in \mathfrak g,\ i=1,2.
\end{equation}
We observe that~\eqref{eq-star-UEA-1} has been proved by different methods in~\cite{Kath} and~\cite[Subsubsections 3.1.2-3.1.4]{AST}.
\begin{Rem}\label{r-star-UEA}
More precisely, in~\cite[]{Kath}, it had been proved that the simple graphs of type $i)$ contribute to the BCH Lie series $Z(\bullet,\bullet)$, while in~\cite[]{AST}, recalling also~\cite{Sh}, it had been proved that the simple graphs of type $ii)$ contribute to the density function $D(\bullet,\bullet)$. 
\end{Rem}
Let us replace in~\eqref{eq-star-UEA-1} $\pi$ by $t\pi$, for $t$ in the unit interval: we write $\star_t$ for the corresponding product, whence
\[
e^{y_1}\star_t e^{y_2}=D(ty_1,ty_2) e^{Z_t(y_1,y_2)},\ Z_t(y_1,y_2)=\frac{Z(ty_1,ty_2)}t.
\]
It follows directly from~\eqref{eq-star} that $e^{y_1}\star_1 e^{y_2}=e^{y_1}\star e^{y_2}$ and $e^{y_1}\star_0e^{y_2}=e^{y_1}e^{y_2}$, $y_i$ in $\mathfrak g$.

Let us compute the derivative w.r.t.\ $t$ of both sides of~\eqref{eq-star-UEA-1}.

Identity~\eqref{eq-KV-1} implies that (see e.g.~\cite[Lemma 3.2]{KV}) 
\[
\frac{d}{dt}Z_t(y_1,y_2)=\left(\left\langle[y_1,F_t(y_1,y_2)],\partial_{y_1}\right\rangle+\left\langle[y_2,G_t(y_1,y_2)],\partial_{y_2}\right\rangle\right)Z_t(y_1,y_2),
\]
where $F_t(y_1,y_2)=F(ty_1,ty_2)/t$ and similarly for $G_t(y_1,y_2)$.

On the other hand, combining~\cite[Lemma 3.2]{KV} with~\cite[Lemma 3.3]{KV} and observing that $\sqrt{j(\bullet)}$ is $\mathfrak g$-invariant, we get
\[
\begin{aligned}
\frac{d}{dt}D(ty_1,ty_2)&=\frac{1}{2t}\mathrm{tr}_\mathfrak g\!\left(\frac{\mathrm{ad}(ty_1)}{e^{\mathrm{ad}(ty_1)}-1}+\frac{\mathrm{ad}(ty_2)}{e^{\mathrm{ad}(ty_2)}-1}-\frac{\mathrm{ad}(tZ_t(y_1,y_2))}{e^{\mathrm{ad}(tZ_t(y_1,y_2))}-1}-1\right)D(ty_1,ty_2)+\\
&\phantom{=}+\left(\left\langle[y_1,F_t(y_1,y_2)],\partial_{y_1}\right\rangle+\left\langle[y_2,G_t(y_1,y_2)],\partial_{y_2}\right\rangle\right)D(ty_1,ty_2)=\\
&=\mathrm{tr}_\mathfrak g(\mathrm{ad}(y_1)\partial_{y_1} F_t(y_1,y_2))+\mathrm{tr}_\mathfrak g(\mathrm{ad}(y_2)\partial_{y_2} G_t(y_1,y_2))D(ty_1,ty_2)+\\
&\phantom{=}+\left(\left\langle[y_1,F_t(y_1,y_2)],\partial_{y_1}\right\rangle+\left\langle[y_2,G_t(y_1,y_2)],\partial_{y_2}\right\rangle\right)D(ty_1,ty_2),
\end{aligned}
\]
where the second equality is a consequence of~\eqref{eq-KV-2}.

Combining both previous results, we get
\[
\begin{aligned}
\frac{d}{dt}(e^{y_1}\star_t e^{y_2})&=\mathrm{tr}_\mathfrak g(\mathrm{ad}(y_1)\partial_{y_1} F_t(y_1,y_2))+\mathrm{tr}_\mathfrak g(\mathrm{ad}(y_2)\partial_{y_2} G_t(y_1,y_2))D(ty_1,ty_2)(e^{y_1}\star_t e^{y_2})+\\
&\phantom{=}+\left(\left\langle[y_1,F_t(y_1,y_2)],\partial_{y_1}\right\rangle+\left\langle[y_2,G_t(y_1,y_2)],\partial_{y_2}\right\rangle\right)(e^{y_1}\star_t e^{y_2})=\\
&=\left(\left\langle[y_1,F_t(y_1,y_2)],\partial_{y_1}\right\rangle+\mathrm{tr}_\mathfrak g(\mathrm{ad}(y_1)\partial_{y_1} F_t(y_1,y_2))\right)(e^{y_1})\star_t e^{y_2}+\\
&\phantom{=}+e^{y_1}\star_t\left(\left\langle[y_2,G_t(y_1,y_2)],\partial_{y_1}\right\rangle+\mathrm{tr}_\mathfrak g(\mathrm{ad}(y_1)\partial_{y_1} G_t(y_1,y_2))\right)(e^{y_2}).
\end{aligned}
\]
Recalling the computations at the beginning of Subsection~\ref{ss-2-3}, it is not difficult to verify that to the Lie series $F_t$, $G_t$, one may associate two smooth $1$-forms $\Omega_i^\mathrm{KV}$, $i=1,2$, on the unit interval with values in $\mathfrak g\otimes\widehat{\mathrm S}(\mathfrak g^*)$, such that the following identities hold true:
\[
\begin{aligned}
\Omega_1^\mathrm{KV}([\pi,e^{y_1}],e^{y_2})&=\langle[y_1,F_t(y_1,y_2)dt],\partial_{y_1}\rangle(e^{y_1})\otimes e^{y_2}+\mathrm{tr}_\mathfrak g\!\left(\mathrm{ad}(y_1)\partial_{y_1}(F_t(y_1,y_2)dt)\right)e^{y_1}\otimes e^{y_2},\\
\Omega_2^\mathrm{KV}([\pi,e^{y_1}],e^{y_2})&=e^{y_1}\otimes\langle[y_2,G_t(y_1,y_2)dt],\partial_{y_2}\rangle(e^{y_2})+\mathrm{tr}_\mathfrak g\!\left(\mathrm{ad}(y_2)\partial_{y_2}(G_t(y_1,y_2)dt)\right)e^{y_1}\otimes e^{y_2},
\end{aligned}
\]
whence, denoting by $\mathcal T_t(\bullet,\bullet)$ the $t$-dependent bidifferential operator of infinite order $\mathcal T_t(f_1,f_2)=f_1\star_t f_2$, $f_i$ in $A$, we find the homotopy formula 
\begin{equation}\label{eq-KV-cup}
f_1\star f_2-f_1f_2=\int_0^1 \left(\mathcal T_t(\Omega_1^\mathrm{KV}([\pi,f_1],f_2))+\mathcal T_t(\Omega_2^\mathrm{KV}(f_2,[\pi,f_1]))\right),
\end{equation}
which is similar in its structure to the homotopy formula~\eqref{eq-cup-comp} obtained by deforming the product $\star$ on $C_{2,0}^+$.


\end{document}